\documentclass[journal,10pt]{IEEEtran}

\usepackage{cite}

\usepackage{xcolor}

\usepackage[cmex10]{amsmath}
\usepackage{amssymb}	
\usepackage{amsthm}							
\newtheorem{assumption}{Assumption}	
	
\newtheorem{proposition}{Proposition}

\theoremstyle{remark}
\newtheorem{remark}{Remark}	
\usepackage{mathtools}
\usepackage{algorithm}
\usepackage{algorithmic}
\usepackage[hyphens]{url}

\usepackage{tikz}
\usepackage{pgfplots}
\usepgfplotslibrary{fillbetween}
\usepgfplotslibrary{patchplots}

\usepackage{comment}

\begin{document}

\title{Optimal Power Allocation in Battery/Supercapacitor Electric Vehicles using Convex Optimization}

\author{Sebastian~East, \and Mark~Cannon
\thanks{S.~East and M.~Cannon are with the Department
of Engineering Science, University of Oxford, Oxford, OX1 3PJ. email: mark.cannon@eng.ox.ac.uk.}
}

\markboth{}%
{}

\maketitle

\begin{abstract}
This paper presents a framework for optimizing the power allocation between a battery and supercapacitor in an electric vehicle energy storage system. A convex optimal control formulation is proposed that minimizes total energy consumption whilst enforcing hard constraints on power output and total energy stored in the battery and supercapacitor. An alternating direction method of multipliers (ADMM) algorithm is proposed, for which computational and memory requirements scale linearly with the length of the prediction horizon (and can be reduced using parallel processing). The optimal controller is compared with a low-pass filter against an all-battery baseline in numerical simulations, where it is shown to provide significant improvement in battery degradation (inferred through reductions of 71.4\% in peak battery power, 21.0\% in root-mean-squared battery power, and 13.7\% in battery throughput), and a reduction of 5.7\% in energy consumption. It is also shown that the ADMM algorithm can solve the optimization problem in a fraction of a second for prediction horizons of more than 15 minutes, and is therefore a promising candidate for online receding-horizon control.
\end{abstract}


%
\IEEEpeerreviewmaketitle

\bstctlcite{IEEEexample:BSTcontrol}

\section{Introduction}
\IEEEPARstart{E}{lectric} vehicles have surged in popularity in the past decade, and are expected to achieve a passenger vehicle market share of 10\% by 2024 \cite[p.3]{Deloitte2019}. One of the challenges facing large scale adoption of electric vehicles, however, is the high cost of lithium-ion batteries, which can reach 50\% of the total cost of a vehicle \cite[\S 3]{Fries2017}. This issue is exacerbated by the batteries' low power density and cycle life ($\leq$2000 W/kg and $\sim$2000 cycles \cite[Table B1]{Zakeri2015}), which imply that the battery may need to be sized above its total energy requirement in order to meet its power requirement, and also require replacement at a higher frequency than other powertrain components.

A hybrid energy storage system consisting of a conventional battery and a supercapacitor can be used to reduce the effect of these limitations. The operational principle is that short term fluctuations and large spikes in power demand are delivered by the supercapacitor whilst the battery delivers power at a reduced and more constant level (see \cite{Ju2016} for a review of battery/supercapacitor vehicle architectures). Supercapacitors are appropriate for this architecture as they have a power density of up to 23,500 W/kg and an effectively infinite cycle life \cite[Table B1]{Zakeri2015}).

One of the factors that determines the effectiveness of the hybrid storage system in reducing battery usage is the control system used to determine the power allocation between the battery and supercapacitor (see \cite[\S 4]{Tie2013} for a comprehensive review of control methods). Rule-based control algorithms are simple and real-time implementable. A common rule-based approach uses a low pass filter to allocate the low frequency components of the power demand signal to the battery and the high frequency components to the supercapacitor \cite{Allegre2013,Nguyen2018,Dusmez2014}. A drawback of these approaches is that they do not account for hardware constraints such as electrical storage capacity, so the vehicle will be forced to allocate charge/discharge demands to the battery during periods when the supercapacitor is completely charged/discharged. Furthermore, it is unlikely that a rule-based controller is optimal (in the sense of minimizing energy consumption or battery degradation) even when hardware limits are not active.

Optimization-based control strategies aim to achieve the best possible power allocation for a given performance criterion. This approach typically has greater computational requirements than rule-based heuristics and assumes knowledge of the predicted power demand over a future horizon, but it is potentially much more effective. Dynamic programming (DP) can be used to determine the globally optimal control inputs for arbitrary cost functions, system dynamics, and hardware constraints. However DP is computationally inefficient, and can typically only be used offline as a benchmark \cite{Masih-Tehrani2013,Santucci2014,Nguyen2018}. A possible use case of DP is to generate target data for tuning real-time control methods; a rule-based controller was tuned to DP results in \cite{Song2015}, and an artificial neural network was trained on DP results in \cite{Shen2015}. A real-time DP control strategy was obtained in \cite{Romaus2010}, where the optimal state-feedback law obtained using stochastic DP was saved as a look-up table, but this approach has a potentially prohibitive memory requirement (12Gb).

Pontryagin's Minimum Principle was investigated in \cite{Nguyen2018} and \cite{Zheng2018} to develop a real-time optimization-based controller, but these problem formulations considered very limited hardware constraints in the optimal control formulation. Hard constraint satisfaction can be addressed using model predictive control (MPC). In \cite{Wang2019,Santucci2014}, and \cite{Hredzak2014} the hybrid energy storage system was modelled as a linear system so that the MPC optimization problem reduced to a quadratic program. Convex quadratic programs can be solved efficiently and reliably using widely available software, however the linear approximation of the system dynamics renders the obtained control inputs suboptimal in general.

In this paper, the limitations of the aforementioned optimization-based control approaches are addressed using convex optimization. This approach allows the use of nonlinear models of powertrain losses and more general system dynamics than linear-quadratic MPC, whilst guaranteeing constraint satisfaction. Furthermore, a convex formulation generally permits a solution in polynomial time \cite{NesterovBook}, and guarantees that a locally optimal solution is also globally optimal. Previous studies have also investigated convex optimization for this application: in \cite{Choi2014} a real-time convex optimization based controller was presented, although constraints on battery power and energy were not considered, and only a one-step prediction horizon was used. Also, a primal-dual interior point algorithm for convex optimization was implemented in \cite{Wang2017}, but this did not optimize the power split in regenerative mode, and the algorithm was only tested on a single drive cycle (49 cycles are used in this paper), and the computational performance of the algorithm was not reported.

\subsection{Contributions}

This paper makes three novel contributions to the electric vehicle energy management problem using convex optimization.
\begin{enumerate}
\item A general convex optimization framework is presented for optimal electric vehicle power allocation. The framework considers losses in each of the battery, supercapacitor, and powertrain, and enforces hard constraints on instantaneous power delivered by the battery, supercapacitor, and powertrain, and total energy stored in both that battery and supercapacitor.
\item A computationally efficient alternating direction method of multipliers (ADMM) \cite{Wang2014} algorithm is presented for the solution of the convex optimization problem. The algorithm is designed to exploit the separability of the problem, and the iteration and memory costs are $\mathcal{O}(T)$  (where $T$ is the length of the prediction horizon). If parallel processing is available, the computation of a subset of the variable updates reduces to $\mathcal{O}(1)$.
\item A set of numerical experiments  are presented in which the proposed ADMM algorithm is compared with a low-pass filter against an all-battery baseline on 49 examples of real driver behaviour. It is demonstrated that the convex formulation significantly reduces several metrics of battery use (Root-mean-square (RMS) battery power, peak battery power, total power throughput, and energy consumption) relative to the low-pass filter, and that the ADMM algorithm solves the convex optimization problem in an average of 0.38s (even for prediction horizons of up to 1003 samples), compared to an average of 63s using general purpose convex optimization software CVX \cite{CVX,gb08}.
\end{enumerate}

The work presented in this paper builds on the authors' previous work developing fast optimization algorithms for energy management in plug-in hybrid electric vehicles (PHEVs). A similar ADMM algorithm was proposed in \cite{EastCDC} and subsequently developed in \cite{EastTCST} and \cite{EastCSL}, but the structure of this problem is different and all of the work presented here is new.

\subsection{Notation}
Vectors of ones and zeros are denoted $\mathbf{1}$ and $\mathbf{0}$ respectively (in all cases the dimension of the vector can be inferred and is not specified). Inequalities are considered element-wise in the context of vectors. The set of integers between upper and lower bounds is denoted $\{ \underline{z}, \dots , \overline{z} \} := \{ z \in \mathbb{Z}: \underline{z} \leq z \leq \overline{z} \}$. 
The nonnegative (positive) real numbers are denoted $\mathbb{R}_+ (\mathbb{R}_{++})$.

\section{Problem Formulation}

\subsection{Variables}

\tikzstyle{input} = [coordinate]
\tikzstyle{sum} = [draw, fill=black!10, circle, node distance=0.5cm]
\tikzstyle{block} = [draw, fill=black!10, rectangle, 
    minimum height=3em, minimum width=5em]

\begin{figure}
\begin{tikzpicture}[auto, node distance=1.5cm,>=latex]
\node [name = input] {$u_t$};
\node [block, name=battery, right of=input, node distance = 2cm, text width=2cm, align=center] {Battery Losses $g_t$};
\node [sum, right of=battery, node distance = 2cm, name=sum] {};
\node [block, right of = sum, name=motor, node distance=1.75cm, text width=1.6cm, align=center] {Powertrain Losses $h_t$};
\node [block, below of=battery, node distance=1.5cm, name=supercaploss, text width=2.1cm, align=center] {Supercapacitor Losses $f_t$};
\node[left of=supercaploss, node distance=2cm, name=supercapinput] {$v_t$};
\node [sum, right of=motor, node distance = 1.75cm, name=sum2] {};
\node [below of =sum2, node distance=1cm, name=brake] {$b_t$};
\node [right of =sum2, node distance=1cm, name=drive] {$d_t$};
\draw [<->] (input) -- (battery);
\draw [<->] (battery) -- (sum) node [pos=0, above right] {$\breve{u}_t$};
\draw [<->] (motor) -- (sum) node [midway, above] {};
\draw [<->] (supercaploss) -| (sum) node [pos=0, above right] {$\breve{v}_t$};
\draw [<->] (supercapinput) -- (supercaploss);
\draw [<->] (motor) -- (sum2) node [midway, above] {$m_t$};
\draw [->] (sum2) -- (brake);
\draw [<->] (sum2) -- (drive);
\end{tikzpicture}
\caption{Power-flow diagram for electric vehicle with hybrid energy storage system.}
\label{figure::powertrain_diagram}
\end{figure}
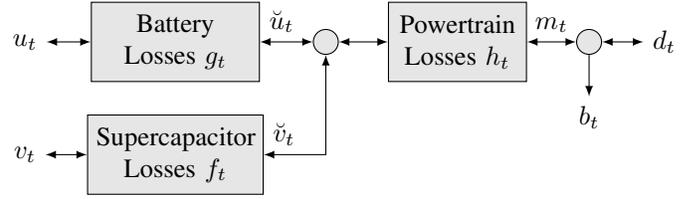

Figure \ref{figure::powertrain_diagram} shows a simplified diagram of the power flows in the powertrain used to formulate the energy management problem, where it is assumed that the storage system is an active architecture in which the power delivered by the battery and supercapacitor can be controlled individually (a discussion of possible power electronics is presented in \cite{Ju2016}). The powertrain is modelled in discrete time with an assumed sampling interval of $1 \,$s, so that $t \in \{0, \dots , T \}$, where $T \in \{1, \dots , \infty \}$ is the duration of the journey under consideration (although the methods presented here can be readily extended to an arbitrary sampling interval). The variable $u := (u_0,\dots,u_{T-1}) \in \mathbb{R}^T$ represents the rate of change of the internal energy of the battery, and by assuming that the power values are constant between sampling intervals, the energy stored in the battery, $x := (x_1,\dots,x_T) \in \mathbb{R}^T$, is given by
$$
x_t := x_0 - \sum_{i=0}^{t-1} u_i \quad \text{for} \quad t \in \{1, \dots, T\},
$$
where $x_0\in \mathbb{R}$ is the battery's internal energy at the start of the journey. Similarly, the variable $v := (v_0, \dots, v_{T-1}) \in \mathbb{R}^T$ represents the rate of change of internal energy in the supercapacitor, so the energy stored in the supercapacitor, $y := (y_1, \dots, y_T) \in \mathbb{R}^T$, is given by
$$
y_t:= y_0 - \sum_{i=0}^{t-1} v_i \quad \textrm{for} \quad t \in \{1, \dots , T \},
$$
where $y_0 \in \mathbb{R}$ is the supercapacitor's internal energy at the start of the journey. The variable $\breve{u} :=(\breve{u}_0,\dots, \breve{u}_{T-1})\in \mathbb{R}^T$ represents the electrical power delivered by the battery after losses, $\breve{v} := (\breve{v}_0, \dots , \breve{v}_{T-1} ) \in \mathbb{R}^T$ is the electrical power delivered by the supercapacitor after losses, $m:=(m_0,\dots m_{T-1}) \in \mathbb{R}^T$ is the mechanical power delivered by the powertrain to the wheels, $b:=(b_0, \dots, b_{T-1}) \in \mathbb{R}_{+}^T$ is the mechanical braking power, and $d:=(d_0, \dots, d_{T-1})$ represents the power demanded by the driver throughout the journey. Methods for predicting future driver behaviour are still an open problem \cite{Zhou2019} that is outside of the scope of this paper; it is assumed here that an exact prediction of $d$ is available to the powertrain controller. The limitations introduced by an inaccurate prediction are discussed in Section \ref{section::numerical_experiments}.

\subsection{Loss Functions}

Battery losses are represented by the time-varying function $g_t$, while powertrain losses are modelled with the time varying function $h_t$, and the supercapacitor losses are modelled with the time varying function $f_t$. It is assumed that the losses in the battery, supercapacitor, and powertrain are independent of states such as temperature and state of charge (this assumption typically introduces minor approximation errors, and it is common in similar energy management problems, e.g. \cite{Hu2019,MURGOVSKI2012106,Nesch2014}). Under the above assumptions, the electrical power delivered by the battery is given by $\breve{u}_t := g_t ( u_t) \ \forall t \in \{0, \dots , T-1 \}$ and the electrical power delivered by the supercapacitor is $\breve{v}_t := f_t ( v_t) \ \forall t \in \{0, \dots , T-1 \}$. These are then combined additively and delivered to the powertrain so that $m_t := h_t ( g_t ( u_t ) + f_t ( v_t ) ) \ \forall t \in \{0, \dots , T-1 \}$. Finally, the power delivered by the powertrain is combined additively with the power from the brakes to meet the driver demand power,
\begin{equation}\label{equation::dynamic_constraint}
d_t := b_t + h_t ( g_t (u _t ) + f_t ( v_t ) )  \quad \forall t \in \{0, \dots , T-1 \}.
\end{equation}
\begin{assumption}\label{assumption::increasing}
For all $t \in \{0, \dots , T-1 \}$, $f_t (\cdot)$, $g_t(\cdot)$, and $h_t(\cdot)$ are strictly increasing functions of their arguments.
\end{assumption}
\begin{assumption}\label{assumption::concave}
For all $t \in \{0, \dots , T-1 \}$, $f_t(\cdot)$ and $g_t(\cdot)$ are concave functions of their arguments.
\end{assumption}
Assumptions \ref{assumption::increasing} and \ref{assumption::concave} are justified by a physical interpretation of a loss function: it would be expected that an increase in output power would require an increase in input power, which implies Assumption \ref{assumption::increasing}, and it would be expected that the losses would increase as the magnitude of the input/output power increases, which implies Assumption \ref{assumption::concave}. In order for the modelled system to be valid, it is also required that $f_t(x) \leq x$, $g_t(x) \leq x$, and $h_t(x) \leq x$ $\forall x,t$ (i.e. the system cannot create energy), but this condition is not explicitly required to obtain a convex formulation.

\begin{assumption}\label{assumption::surjective}
For all $t \in \{0, \dots , T-1 \}$, $f_t(\cdot)$, $g_t(\cdot)$, and $h_t(\cdot)$ are surjective functions.
\end{assumption}

Assumption \ref{assumption::surjective} is a technicality required for Assumption \ref{assumption::increasing} to imply that $f_t(\cdot)$, $g_t(\cdot)$, and $h_t(\cdot)$ are bijective $\forall t \in \{0, \dots , T-1 \}$, which therefore implies that $f_t^{-1}(\cdot)$, $g^{-1}_t(\cdot)$, and $h_t^{-1}(\cdot)$ exist $\forall t \in \{0, \dots , T-1 \}$.

Note that whilst the focus of this paper is on battery/supercapacitor electric vehicles, the proposed approach can be used for any hybrid energy storage system in which the loss functions satisfy Assumptions \ref{assumption::increasing}-\ref{assumption::surjective}.

\subsection{Optimal Control Problem}

The power delivered by the battery, power delivered by the supercapacitor, total electrical power delivered to the powertrain, total battery energy, and total supercapacitor energy are all subject to upper and lower bounds, and the braking power is constrained to be nonpositive:
\begin{gather*}
\underline{u}_t \leq u_t \leq \overline{u}_t, \quad \underline{v}_t \leq v_t \leq \overline{v}_t, \quad \underline{e}_t \leq g_t ( u_t ) + f_t ( v_t ) \leq \overline{e}_t, \\
\underline{x}_t \leq x_t \leq \overline{x}_t, \quad \underline{y}_t \leq y_t \leq \overline{y}_t, \quad b_t \leq 0,
\end{gather*}
$\forall t \in \{0,\dots,T-1\}.$ The cost function
\begin{equation*}\label{equation::cost_function}
\sum_{t=0}^{T-1} \left[ u_t + v_t - h_t(g_t (u_t) + f_t ( v_t ) ) - b_t \right]
\end{equation*}
is used to represent the total losses in the system (although any convex function of $u$, $v$, $x$, and/or $y$ can be used in the proposed framework). This cost represents the sum over $T$ steps of the braking energy, $-b_t$, plus the difference between the combined energy delivered by the battery and supercapacitor, $(u_t + v_t)$, and the total energy delivered by the powertrain, $h_t(g_t(u_t) + f_t(v_t))$. 

Therefore, the optimal control sequences, denoted $(u,v,b)^\star$, are obtained from the minimizing argument of
\begin{equation}\label{equation::optimal_control_problem}
\begin{aligned}
\underset{(u,v,b)}{\textrm{min}} \ & \sum_{t=0}^{T-1} \left[ u_t + v_t - h_t(g_t (u_t) + f_t(v_t)) - b_t \right] \\
\textrm{s.t.} \ & \begin{rcases}
d_t = b_t + h_t ( g_t (u _t ) + f_t ( v_t ) ) \\
b_t \leq 0 \\
 \underline{u}_t \leq u_t \leq \overline{u}_t \\
\underline{v}_t \leq v_t \leq \overline{v}_t \\
\underline{e}_t \leq g_t(u_t) + f_t ( v_t ) \leq \overline{e}_t \\
x_{t+1} = x_0 - \sum_{i=0}^t u_i \\
\underline{x}_{t+1} \leq x_{t+1} \leq \overline{x}_{t+1} \\
y_{t+1} = y_0 - \sum_{i=0}^t v_i \\
\underline{y}_{t+1} \leq y_{t+1} \leq \overline{y}_{t+1} \end{rcases} \forall t \in \{ 0, \dots , T-1 \}.
\end{aligned}
\end{equation}
Problem (\ref{equation::optimal_control_problem}) is generally nonconvex if any of $f_t(\cdot)$, $g_t(\cdot)$, or $h_t(\cdot)$ are nonlinear for any $t \in \{0,\dots,T-1\}$.

\subsection{Convex Formulation}

The constraint $b_t \leq 0 \ \forall t \in \{0, \dots, T-1 \}$ can be combined with (\ref{equation::dynamic_constraint}) to obtain $d_t \leq h_t(g_t(u_t) + f_t( v_t ))$, which under Assumption \ref{assumption::increasing} is equivalent to
$$
h_t^{-1}(d_t) \leq g_t(u_t) + f_t (v_t ).
$$
This is combined with the constraint $\underline{e}_t \leq g_t (u_t ) + f_t ( v_t )$ as
$$
\hat{e}_t \leq g_t (u_t ) + f_t ( v_t )
$$
where $\hat{e}_t := \max \{ \underline{e}_t, h_t^{-1}(d_t) \}$. The set $\{(u_t,v_t) \in \mathbb{R}^2: \hat{e}_t \leq g_t (u_t ) + f_t ( v_t ) \}$ is convex under Assumption \ref{assumption::concave}.  The constraint $g_t(u_t ) + f_t (v_t ) \leq \overline{e}_t$ is in general nonconvex, and is linearly approximated by
\begin{gather*}\label{equation::linearization}
\hat{g}_t(u_t) + \hat{f}_t ( v_t ) \leq \overline{e}_t , \quad \hat{g}_t(u_t) := g_t(\hat{u}_t) + g'_t(\hat{u}_t)(u_t - \hat{u}_t), \\
\hat{f}_t(v_t) := f_t(\hat{v}_t) + f'_t(\hat{v}_t)(v_t - \hat{v}_t),
\end{gather*}
where $g'_t(\cdot )$ and $f'_t(\cdot) $ are the derivatives of $g_t(\cdot)$ and $f_t(\cdot)$\footnote{If $f_t(\cdot)$ or $g_t(\cdot)$ are nondifferentiable, then any non-zero element from their sub-gradients can be used in place of $f_t'(\cdot)$ and $g_t'(\cdot)$.}, and $\hat{u}_t \in \mathbb{R}$ and $\hat{v}_t \in \mathbb{R}$ are fixed linearization points\footnote{The choice of $\hat{u}_t$ and $\hat{v}_t$ is discussed further in Remark \ref{remark::remark1} in Appendix \ref{appendix::combined_u_v_algorithm}.} chosen such that $g_t(\hat{u}_t) \neq 0$ and $f_t (\hat{v}_t ) \neq 0$.
The concavity of $g_t$ and $f_t$ (Assumption \ref{assumption::concave}) implies that
$\{(v_t,u_t) \in \mathbb{R}^2 : \hat{g}_t (u_t ) + \hat{f}_t (v_t ) \leq \overline{e}_t \} \subseteq \{(v_t,u_t) \in \mathbb{R}^2 : {g}_t (u_t ) + f_t ( v_t ) \leq \overline{e}_t \}$, so the linear approximation guarantees that the original constraint is enforced, and the set $\{(v_t,u_t) \in \mathbb{R}^2 : \hat{g}_t (u_t ) + \hat{f}_t (v_t ) \leq \overline{e}_t \} $ is convex. This approximation could introduce some conservatism to the problem formulation. However, $g_t(u_t ) + f_t ( v_t ) \leq \overline{e}_t$ is the upper bound on powertrain power (typically imposed by the torque limits of the powertrain components), and, assuming that the driver demand power cannot exceed this upper limit, the constraint $g_t(u_t)+f_t(v_t) \leq \overline{e}_t$ can only be active if the brakes are active while the battery and supercapacitor deliver power in excess of the demand power. Given that the cost function penalizes energy loss, it is unlikely that this constraint is active at the optimal solution of (\ref{equation::optimal_control_problem}), although it is technically possible in particular operating conditions  (the issue is discussed further in Appendix \ref{appendix::analysis_of_optimization_problem}). 

Using (\ref{equation::dynamic_constraint}), the objective of (\ref{equation::optimal_control_problem}) can be simplified:
$$
u_t + v_t - h_t(g_t (u_t) + f_t ( v_t ) ) - b_t = u_t + v_t - d_t,
$$
where $d$ is independent of the decision variables. Problem (\ref{equation::optimal_control_problem}) can therefore be approximated by the convex optimization problem
\begin{equation}\label{equation::optimal_control_problem_convex}
\begin{aligned}
\underset{(u,v)}{\textrm{min}} \ & \mathbf{1}^\top (u + v) \\
\textrm{s.t.} \ &  x = \mathbf{1}x_0 - \Psi u, \quad x \in \mathcal{X}, \\
& y = \mathbf{1}y_0 - \Psi v, \quad y \in \mathcal{Y}, \\
& \begin{rcases} (u_t,v_t) \in \mathcal{C}_t \\
 u_t \in \mathcal{U}_t \\
 v_t \in \mathcal{V}_t \end{rcases} \ \forall t \in \{ 0, \dots , T-1 \},
\end{aligned}
\end{equation}
where $\Psi$ is a $T \times T$ lower triangular matrix of ones, and
\begin{align*}
\mathcal{C}_t := \ & \{(u_t,v_t) \in \mathbb{R}^{2}: \hat{e}_t \leq g_t(u_t) + f_t (v_t ), \   \\
& \quad \hat{g}_t (u_t) + \hat{f}_t ( v_t ) \leq \overline{e}_t \}, \\
\mathcal{U}_t := \ & \{ u_t \in \mathbb{R} :\underline{u}_t \leq u_t, \leq \overline{u}_t \}, \\
\mathcal{V}_t := \ & \{ v_t \in \mathbb{R} :\underline{v}_t \leq v_t, \leq \overline{v}_t  \}, \\
\mathcal{X} := \ & \{ x \in \mathbb{R}^T :\underline{x}_t \leq x_t, \leq \overline{x}_t \ \forall t \in \{1,\dots,T\} \}, \\
\mathcal{Y} := \ & \{ y \in \mathbb{R}^T :\underline{y}_t \leq y_t, \leq \overline{y}_t \ \forall t \in \{1,\dots,T\} \}.
\end{align*}
The optimal braking power sequence, $b^\star$, is then obtained from $b_t^\star = d_t - h_t(g_t(u_t^\star ) + f_t ( v_t^\star ) ) \ \forall t \in \{0,\dots , T-1 \}$, where $u^\star$ and $v^\star$ are the minimizing arguments of (\ref{equation::optimal_control_problem_convex}).

\section{Alternating Direction Method of Multipliers}\label{section::ADMM}

Problem (\ref{equation::optimal_control_problem_convex}) can be solved using general purpose convex optimization software (e.g. CVX), but these tools are typically computationally intensive. In this section an ADMM algorithm is proposed that is tailored to the structure of (\ref{equation::optimal_control_problem_convex}).  

\subsection{Algorithm}\label{section::algorithm}
Problem (\ref{equation::optimal_control_problem_convex}) is equivalent to the equality constrained problem
\begin{equation}\label{equation::optimal_control_problem_convex_2}
\begin{aligned}
\underset{(u,v)}{\textrm{min}} \ & \mathbf{1}^\top (u + v) + \sum_{t=0}^{T-1}  \left[ \mathcal{I}_{\mathcal{C}_t} (u_t, v_t) + \mathcal{I}_{\mathcal{U}_t}(u_t) + \mathcal{I}_{\mathcal{V}_t}(v_t) \right] \\
& + \mathcal{I}_{\mathcal{X}}(x) + \mathcal{I}_{\mathcal{Y}}(y), \\
\textrm{s.t.} \ & u = \zeta, \quad  v = \eta, \quad x = \mathbf{1}x_0 - \Psi \zeta, \quad  y = \mathbf{1}y_0 - \Psi \eta,
\end{aligned}
\end{equation}
where $\zeta \in \mathbb{R}^T$ and $\eta \in \mathbb{R}^T$ are vectors of dummy variables, and the indicator functions are defined for a given set $\mathcal{S}$ by
$$
\mathcal{I}_\mathcal{S} (s) := \begin{cases} 0 & s \in \mathcal{S}, \\ \infty & \textrm{otherwise.} \end{cases}
$$
Problem (\ref{equation::optimal_control_problem_convex_2}) is in turn the equivalent of
\begin{equation}\label{equation::ADMM_standard_form}
\begin{aligned}
\underset{(u,v)}{\textrm{argmin}} \ \tilde{f}(\tilde{u}) \quad \textrm{s.t.} \ A \tilde{u} + B \tilde{x} = c,
\end{aligned}
\end{equation}
where
\begin{gather*}
\tilde{u}:=(u,v, x, y), \quad \tilde{x}:=(\zeta, \eta), \\
\begin{split}
\tilde{f}(\tilde{u}) := \mathbf{1}^\top (u + v) + \sum_{t=1}^{T-1} [ \mathcal{I}_{\mathcal{C}_t}(u_t, v_t ) & +  \mathcal{I}_{\mathcal{U}_t}(u_t) + \mathcal{I}_{\mathcal{V}_t}(v_t)] \\
& + \mathcal{I}_\mathcal{X}(x) + \mathcal{I}_\mathcal{Y}(y), 
\end{split}
\\
A := \left[ \begin{smallmatrix} I \\  & I \\ & & I \\ & & & I \end{smallmatrix} \right], \quad B:= \left[ \begin{smallmatrix} -I \\ & -I \\  \Psi \\  & \Psi \end{smallmatrix} \right], \quad c:= (\mathbf{0},\mathbf{0},\mathbf{1}x_0, \mathbf{1}y_0 ).
\end{gather*}
We define the augmented Lagrangian function for (\ref{equation::ADMM_standard_form}) as
$$
\mathcal{L}(\tilde{u},\tilde{x},\lambda ) := \tilde{f}(\tilde{u} ) + \frac{1}{2} \|A \tilde{u} + B \tilde{x} - c + \lambda \|^2_\rho 
$$
with
\begin{gather*}
\| x \|^2_\rho:= x^\top \rho x, \quad \rho := \textrm{diag} (\rho_1 \textbf{1}, \rho_2 \textbf{1}, \rho_3 \textbf{1}, \rho_4 \textbf{1}), \\
\quad \lambda := (\lambda_1,\lambda_2,\lambda_3,\lambda_4) .
\end{gather*}
Here $\lambda_1,\lambda_2,\lambda_3,\lambda_4\in\mathbb{R}^T$ are Lagrange multipliers for the constraints in~(\ref{equation::optimal_control_problem_convex_2}), and $\rho_1,\rho_2,\rho_3,\rho_4 \in \mathbb{R}_{++}$ are design parameters (the choice of $\rho$ is discussed in Section~\ref{subsection::parameters}).
The ADMM iteration is defined by
\begin{subequations}\label{equation::ADMM_iteration}
\begin{align}
\tilde{u}^{(j+1)} & := \underset{\tilde{u}}{\textrm{argmin}} \ \mathcal{L}(\tilde{u}, \tilde{x}^j, \lambda^j), \label{subequation::u_update} \\
\tilde{x}^{(j+1)} & := \underset{\tilde{x}}{\textrm{argmin}} \ \mathcal{L}(\tilde{u}^{(j+1)}, \tilde{x}, \lambda^j), \label{subequation::x_update} \\
\lambda^{(j+1)} & := \lambda^j + A \tilde{u}^{(j+1)} + B \tilde{x}^{(j+1)} - c, \label{subequation::lambda_update}
\end{align}
\end{subequations}
and in \cite{Wang2014} it was proved that (\ref{equation::ADMM_iteration}) will converge to the solution of (\ref{equation::ADMM_standard_form}) as the residuals defined by
\begin{align*}
r^{(j+1)} & := A \hat{u}^{(j+1)} + B \hat{x}^{(j+1)} - c\\
s^{(j+1)} & := A^\top \rho B ( \hat{x}^{(j+1)} - \hat{x}^{(j)} )
\end{align*}
necessarily converge to zero (the proof is presented for a positive scalar value of $\rho$, but is trivially extended to a fixed positive diagonal matrix as presented here). The algorithm is initialized with the values 
$$
\hat{u}^{(0)} = \mathbf{0}, \quad \hat{x}^{(0)} = \mathbf{0}, \quad \lambda^{(0)} = \mathbf{0},
$$
and terminated on satisfaction of the criterion
\begin{equation}\label{equation::termination_criterion}
\max \{ \|r^{(j+1)} \|, \|s^{(j+1)} \| \} \leq \epsilon ,
\end{equation}
where $\epsilon \in \mathbb{R}_{++}$ is a pre-determined convergence threshold.

\subsection{Variable Updates \& Algorithm Complexity}
Update (\ref{subequation::u_update}) is equivalent to
\begin{subequations}
\begin{align}
(u_t, v_t)^{(j+1)} & := \underset{(u_t ,v_t)}{\textrm{argmin}} \Big[ \mathcal{I}_{\mathcal{C}_t} (u_t, v_t) + \mathcal{I}_{\mathcal{U}_t}(u_t) + \mathcal{I}_{\mathcal{V}_t}(v_t) + \notag \\
& \ \frac{\rho_1}{2} (u_t - \zeta_t^{(j)} + \lambda_{1,t}^{(j)})^2 + \frac{\rho_2}{2} ( v_t - \eta_t^{(j)} + \lambda_{2,t}^{(j)})^2  \Big] \label{subequation::u_v_update} \\
x^{(j+1)} & := \Pi_\mathcal{X} \left[ \textbf{1} x_0 - \Psi \zeta^{(j)} - \lambda_3^{(j)} \right], \label{subequation::x_update_2} \\
y^{(j+1)} & := \Pi_\mathcal{Y} \left[ \textbf{1} y_0 - \Psi \eta^{(j)} - \lambda_4^{(j)} \right], 
\end{align}
\end{subequations}
where the projection onto a set $\mathcal{S}$ is given by $\Pi_{\mathcal{S}} (s) := \textrm{argmin}_{\hat{s} \in \mathcal{S}} \| s - \hat{s} \|_2$. The combined $(u_t,v_t)$ update in (\ref{subequation::u_v_update}) is a convex optimization problem subject to inequality constraints, and a method is presented in Appendix \ref{appendix::combined_u_v_algorithm} for solving this problem from a finite set of candidate solutions, thereby avoiding the use of a general purpose inequality constrained convex optimization algorithm (e.g. interior-point algorithm), which could increase the computational complexity of the ADMM algorithm as a whole. The update for $(u,v)$ is separable w.r.t.\ each element $(u_t,v_t)$, and therefore its computational complexity scales linearly with $T$ when each update is performed sequentially for $t \in \{0, \dots , T-1 \}$, or is independent of $T$ when they are performed in parallel. The computation of the argument of the projection in (\ref{subequation::x_update_2}) scales linearly with $T$ because multiplication by $\Psi$ is equivalent to a cumulative sum (and has no memory requirement). The projection $\Pi_\mathcal{X}$ can then be performed element-wise, so (\ref{subequation::x_update_2}) scales linearly with $T$ overall, and the same is true for $y^{(j+1)}$.

Update (\ref{subequation::x_update}) is equivalent to
\begin{subequations}
\begin{align}
\zeta^{(j+1)} & := (\rho_1 I + \rho_3 \Psi^\top \Psi )^{-1} \Big[ \rho_1 ( u^{(j+1)} + \lambda_1^{(j)} ) \notag \\
& \quad \quad \quad - \rho_3\Psi^\top ( x^{(j+1)} - \mathbf{1} x_0 + \lambda_3^{(j)}) \Big] \label{subequation::zeta_update} \\
\eta^{(j+1)} & := (\rho_2 I + \rho_4 \Psi^\top \Psi )^{-1} \Big[ \rho_2 ( v^{(j+1)} + \lambda_2^{(j)}) \notag \\
& \quad \quad \quad - \rho_4 \Psi^\top (y^{(j+1)} -\mathbf{1} y_0 + \lambda_4^{(j)} ) \Big] \label{subequation::eta_update}
\end{align}
\end{subequations}
where equations (\ref{subequation::zeta_update}) and (\ref{subequation::eta_update}) are the solutions of the general system of linear equations $(kI + \Psi^\top \Psi)\mathbf{x} = \mathbf{b}$. In Appendix~\ref{appendix::ADMM_linear_system} it is demonstrated that these solutions can be obtained with $\mathcal{O}(T)$ computation and memory requirement.

The residual updates are the equivalent of 
\begin{align*}
r^{(j+1)} = \left[ \begin{smallmatrix}
u^{(j+1)} - \zeta^{(j+1)} \\
v^{(j+1)} - \eta^{(j+1)} \\
x^{(j+1)} + \Psi \zeta^{(j+1)} - \textbf{1} x_0 \\
y^{(j+1)} + \Psi \eta^{(j+1)} - \textbf{1} y_0
\end{smallmatrix} \right] \  s^{(j+1)} = \left[ \begin{smallmatrix}
\rho_1 ( \zeta^{(j)} - \zeta^{(j+1)} ) \\
\rho_2 ( \eta^{(j)} - \eta^{(j+1)} ) \\
\rho_3 \Psi (\zeta^{(j+1)} - \zeta^{(j)} ) \\
\rho_4 \Psi (\eta^{(j+1)} - \eta^{(j)} )
\end{smallmatrix} \right]
\end{align*}
which scale linearly with $T$ and have no additional memory requirement (multiplication by $\Psi$ is the equivalent of a cumulative sum). Finally, update (\ref{subequation::lambda_update}) is equivalent to
\begin{align*}
\lambda^{(j+1)} := \lambda^{(j)} + r^{(j+1)}.
\end{align*}
The overall scaling properties w.r.t.\ horizon length $T$ of the computation of each iteration are summarised in Table 1\footnote{Note that the scaling properties presented here are an improvement relative to those presented in \cite{EastTCST} and \cite{EastCSL} due to the result of Appendix \ref{appendix::ADMM_linear_system}.}, and the memory requirement of the algorithm is $\mathcal{O}(T)$, as only two bandwidth-2 matrices (defined in Appendix \ref{appendix::ADMM_linear_system}) and the variables themselves require storage. The ADMM iteration (\ref{equation::ADMM_iteration}) is a particular case of `Generalized ADMM', for which it was demonstrated in \cite{Adona2017} that $\mathcal{O}(1/\epsilon )$ iterations are required to meet criterion (\ref{equation::termination_criterion}) for a given problem. To the best of the authors' knowledge there are no results that demonstrate how the iteration complexity varies with horizon length (i.e. problem size), but numerical studies using ADMM for the PHEV energy management problem suggest that the number of iterations required for a given tolerance $\epsilon$ is independent of horizon length \cite[\S 5-C]{EastTCST}.

\begin{table}
\begin{center}
\caption{ Summary of the complexity of each of the ADMM variable updates w.r.t.\ horizon length $T$.}
\label{table::ADMM_complexity}
\begin{tabular}{c | c c c c c c c c c c}
$\mathcal{O}(\cdot )$ & $(u , v )$ & $x$ & $y$ & $\zeta$ & $\eta$ & $r$ & $s$ & $\lambda$ \\ \hline
Parallel & 1 & T & T & T & T & T & T & 1  \\
Sequential & T & T & T & T & T & T & T & T 
\end{tabular}
\end{center}
\end{table}

\section{Numerical Experiments}\label{section::numerical_experiments}

This section details the numerical experiments used to validate the performance of the convex optimal control formulation. The parameters of the model used for simulation are provided in Table \ref{table::parameters}.

\begin{table}
\centering
\caption{Model Parameters for a Typical Passenger Vehicle \label{table::parameters}}
\begin{tabular}{c | c}
Parameter & Value \\ \hline
Vehicle mass $(M)$ & 1900 kg \\
Drag coefficient $(C_d)$ & 0.27 \\
Air density $(\rho_a)$ & 1.225 kg/m$^3$\\
Rolling resistance coefficient $(C_r)$ & 0.015 \\
Acceleration due to gravity $(g)$ & 9.81 m/s$^2$ \\
Motor torque limits $(\underline{T}, \overline{T})$ & $\pm$ 250 Nm \\
Battery open-circuit voltage $(V)$ & 300 V \\
Battery open-circuit resistance $(R)$ & 0.1 $\Omega$ \\
Battery power limits $(\underline{P}, \overline{P})$ & $\pm$ 70 kW \\
Battery energy limits $(\underline{x},\overline{x})$ & (0,80) MJ \\
Supercapacitor energy limits $(\underline{y},\overline{y})$ & (0, 1.08) MJ
\end{tabular}
\end{table}

\subsection{Vehicle \& Driver Model}

Figure \ref{figure::velocity} shows velocity and road gradient trajectories used to generate optimization scenarios. These are taken from 49 instances of real drive-test data on a single route, each of which was simulated individually.
\begin{figure}
\begin{center}
\includegraphics[scale=1]{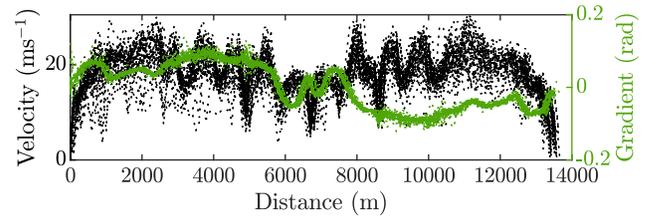}
\caption{Velocity and gradient data against distance (figure taken from \cite{EastCSL}) \label{figure::velocity}}
\end{center}
\end{figure} 
Power demand trajectories were generated from the longitudinal model
$$
d_t = (M \dot{v}_t + \frac{1}{2} \rho_a v_t^2 C_d A + C_r M g \cos \theta_t + M g \sin \theta _t ) v_t,
$$
where at sample $t$, $v_t \in \mathbb{R}$ is the velocity, $\dot{v}_t \in \mathbb{R}$ is the acceleration (obtained using central difference numerical differentiation with a sample period of $1 \ $s), and $\theta_t \in \mathbb{R}$ is the road gradient.

\subsection{Powertrain Model}

The vehicle was modelled with a single-speed transmission so that the rotational speed of the motor, $\omega_m \in \mathbb{R}^T$, was calculated at each timestep from $\omega_{m,t} = \frac{v_t}{r_w r_d} \ \forall t \in \{0, \dots, T-1 \}$, where $r_w \in \mathbb{R}$ is the effective radius of the wheel and $r_d\in \mathbb{R}$ is the final drive ratio of the transmission. The losses in the motor were modelled as a quadratic function mapping the output power, $m_t$, to motor input power, $g_t(u_t) + f_t(v_t)$, with parameters dependent on the motor speed, $\omega_{m,t}$: 
$$
g_t(u_t) + f_t(v_t ) := \beta_{2}(\omega_{m,t}) m_t^2 + \beta_{1}(\omega_{m,t}) m_t + \beta_{0}(\omega_{m,t})
$$
where $\beta_{2}(\omega_{m,t}) > 0 \ \forall \omega_{m,t}$ (this is a common method for approximating the motor loss map, e.g. \cite[\S II. B.]{Hu2019} \cite[\S 7.5]{MURGOVSKI2012106} \cite[\S 2.5]{Nesch2014}). It was assumed that all drivetrain components other than the motor were 100\% efficient, so that the inverse powertrain losses were modelled using the sampled coefficients, $\beta_{i,t} = \beta_i(\omega_{m,t}) \ \forall i \in \{0,1,2\}$, for each sample of $\omega_{m,t}$ as
$$
h_t^{-1}(m_t) := \beta_{2,t} m_t^2 + \beta_{1,t} m_t + \beta_{0,t} \quad \forall t \in \{0, \dots, T-1 \},
$$ 
which is invertible on the domain $\left[ -\frac{\beta_{1,k} }{2 \beta_{2,k}  }, \infty \right]$ and range $ \left[ \beta_{0,k} - \frac{\beta_{1,k}^2}{4 \beta_{2,k}} , \infty \right]$. As a result the powertrain loss function,
$$
h_t(x) := \frac{-\beta_{1,t} + \sqrt{ \beta_{1,t}^2 - 4 \beta_{2,t}(\beta_{0,t} - x)}}{2 \beta_{2,t}},
$$
satisfies assumptions \ref{assumption::increasing} and \ref{assumption::surjective}. The motor was also subject to upper and lower bounds on power due to its lower and upper torque limits, $\underline{T}$ and $\overline{T}$, so the overall power limits were
$$
\underline{e}_k := \max \left\{ \beta_{0,k} - \frac{\beta_{1,k}^2}{4 \beta_{2,k}}, \underline{T} \omega_{m,k} \right\}, \quad \overline{e}_k := \overline{T} \omega_{m,k}.
$$

\subsection{Battery \& Model}\label{subsection::battery_model}

The battery was modelled as an equivalent circuit with constant internal resistance, $R$, and open circuit voltage, $V$ (the assumption that these are constant parameters  holds for modest changes in state-of-charge and is used in similar energy management formulations \cite[\S 7.1]{MURGOVSKI2012106}, \cite[\S 2.7]{Nesch2014}), so that
$$
g_t(u_t) := \frac{V^2 - (V - \frac{2Ru_t}{V} )^2}{4 R}\quad \forall t \in \{0, \dots, T-1 \},
$$
which satisfies Assumptions \ref{assumption::increasing}, \ref{assumption::concave}, and \ref{assumption::surjective} on the domain $\left[ -\infty , \frac{V^2}{2R} \right]$ and range $\left[ -\infty , \frac{V^2}{4 R} \right]$. The linearization point was set at $\hat{u}_t = 0 \ \forall t \in \{0, \dots , T-1 \}$.

The battery was subject to upper and lower bounds on power, $\overline{P}$ and $\underline{P}$ (that equate to current limits under the assumption of constant battery voltage), so that 
$$
\underline{u}_t := \underline{P} \quad \overline{u}_t := \min \left\{  \overline{P},  \frac{V^2}{2R} \right\}
$$
$\forall t \in \{0, \dots, T-1\}$, and the battery upper limit was set to 22 kWh (80 MJ) to model a typical electric vehicle battery capacity (e.g. 2015 Renault Fluence ZE \cite[Table 2]{MAHMOUDZADEHANDWARI2017414}). It was assumed that the battery's power electronics (e.g. DC-DC converter) were 100\% efficient.

\subsection{Supercapacitor Model}\label{subsection::supercap_model}

It was assumed that the supercapacitor and associated power electronics were 100\% efficient so that
$$
f_t(v_t) = v_t \quad \forall t \in \{0, \dots , T-1 \},
$$
which trivially satisfies Assumptions \ref{assumption::increasing}, \ref{assumption::concave}, and \ref{assumption::surjective} (and does not require linearization, so the  choice of $\hat{v}_t$ is arbitrary). The power limits were infinite so that $\underline{v}_t = - \infty$ and $\overline{v}_t = \infty$ $\forall t \in \{0, \dots, T-1\}$. There are currently no commercially available battery/supercapacitor electric vehicles, so the supercapacitor energy limit was set at 300 Wh (1.08 MJ) to reflect the parameters used in similar studies (441.5 Wh was used in \cite[Table 2]{WIECZOREK2017222}; 203 Wh in \cite[\S II.B.]{Shen2015}; and 0.8 MJ in \cite[Table IV]{Romaus2010}).

\subsection{ADMM Parameters}\label{subsection::parameters}

The $\rho$ parameters detailed in Section \ref{section::algorithm} require tuning. A two-dimensional parameter search similar to that detailed in \cite[\S V-B]{EastTCST} was used to optimize the parameters $\rho_1$ and $\rho_3$, with $\rho_2 = \rho_1$ and $\rho_4 = \rho_3$ as these parameters correspond to constraints of similar magnitude. The chosen values, $\rho_1 = 5 \times 10^{-5}$ and $\rho_3 = 1 \times 10^{-8}$, were used for all simulations (the results in \cite{EastTCST} and \cite{EastCSL} suggest that the $\rho$ values are tuned to hardware characteristics, and fixed values perform well across a diversity of drive cycles).

The termination criterion (\ref{equation::termination_criterion}) enforces an upper bound on a measure of constraint violation ($\| r^{(j+1)} \|$) and sub-optimality ($\| s^{(j+1)} \|$). In the experiments presented here the termination criterion $\epsilon = 100$ was used, which can be loosely interpreted as a 0.1\% upper bound on solution error (as $u$ and $v$ took values of the order of magnitude $10^5$). The solutions obtained using CVX and ADMM were indistinguishable using this criterion.

\subsection{Control Algorithms}\label{subsection::algorithms}

Each of the 49 test-drives shown in Figure \ref{figure::velocity} was simulated using three alternative power allocation methods:
\begin{enumerate}
\item \textit{All-battery:} All positive (and negative power) was delivered from (to) the battery, unless the upper bound on the battery energy was active and the power demand was negative, in which case the excess power was delivered by the brakes.
\item \textit{Low-pass Filter:} A first order low-pass filter with a bandwidth of $0.01$ Hz was used to separate the power demand frequencies. The filtered signal was allocated to the battery, and the remaining power demand was allocated to the supercapacitor, unless the supercapacitor limits were active, in which case the excess power was also delivered by the battery. 
\item \textit{Optimal:} The battery and supercapacitor were controlled using the optimal controls obtained from the solution of (\ref{equation::optimal_control_problem_convex}) in open-loop (no uncertainty was modelled in the driver behaviour predictions). The solution was obtained for each journey using both ADMM and CVX to determine the relative computational performance. The ADMM algorithm was programmed in Matlab, the default solver and tolerance was used for CVX, and a 2.60GHz Intel Core i7-9750H CPU was used for both. 
\end{enumerate}
\subsection{Results}

Figure \ref{figure::results2} shows the battery power, supercapacitor power, battery energy, and supercapacitor energy obtained using each of the algorithms detailed in Section \ref{subsection::algorithms} for a single journey. Qualitatively, it can be seen that the low-pass filter reduces the amplitude and frequency of the peaks in the battery control signal relative to the all-battery controller, whilst the optimal controls obtained from the solution of (\ref{equation::optimal_control_problem_convex_2}) result in a piecewise constant battery control signal (this is explained further in Appendix \ref{appendix::analysis_of_optimization_problem2}). The optimal solution suggests that it is challenging to approximate the optimal controls using a linear filter, as the presence of both hard discontinuities and periods of constant output place conflicting requirements on the bandwidth of the filter. The battery power constraint is violated for the all-battery mode as the battery is the only power source, but it is also violated using the low pass filter. In particular, it is violated at two large peaks at $\sim$50 s and $\sim$200s, where the supercapacitor is completely empty and the battery is forced to deliver the positive demand power. Conversely, the battery power constraint is satisfied at all times using the optimal controller, for which the only hard constraints that are active are the upper and lower limits on supercapacitor energy. It was found that by tightening the battery power or energy limits to ensure that they were active at the solution, the problem generally becomes infeasible.
\begin{figure}
\begin{center}
\includegraphics[scale=1]{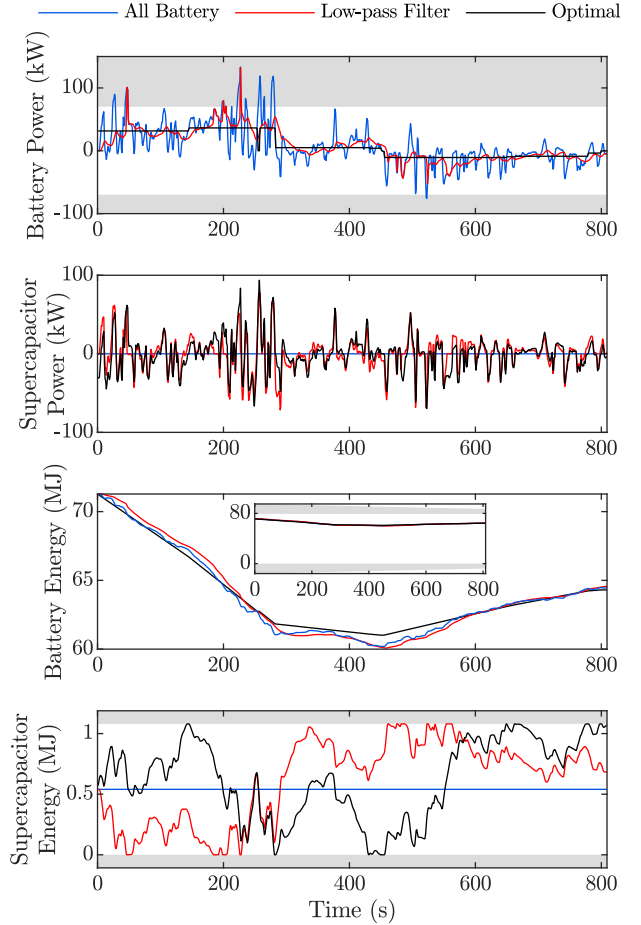}
\caption{Trajectories for battery power, supercapacitor power, battery energy, and supercapacitor energy for a single journey using all-battery mode, a low-pass filter, and optimal controls. Hard constraints are shown in grey areas. \label{figure::results2}}
\end{center}
\end{figure} 

An aim of the energy management optimization was to minimize battery degradation, a complex phenomenon caused by a multitude of factors (see \cite{BARRE2013680} for a comprehensive review of lithium-ion battery ageing mechanisms). There are a range of explicit models of battery degradation (for a review of modelling methods see \cite{Reniers_2019}), but these typically model the battery at a level that is inconsistent with the resolution of the simulations performed here (e.g. the current distribution on a cellular level would depend on the battery management system, which is not modelled in this paper). Therefore, three metrics of battery power were instead used as approximate measures of battery degradation: the Root-Mean-Squared (RMS) battery power, RMS($u$), the peak battery power, $\max | u|$, and total power throughput, $\sum |u |$. Additionally, the total energy consumption, $\sum (u + v)$, was used to determine relative efficiency of each control method (this was the optimization objective in (\ref{equation::optimal_control_problem_convex})). Note that RMS($u$), $\max |u|$, and $\sum |u|$ could also be included as objectives in a convex optimization problem, but the results show that the proposed objective function is a good heuristic for minimizing all four quantities. 
Table \ref{table::results1} shows each of these measures for the trajectories in Figure \ref{figure::results2}, and it can be seen that the optimal controller provides a significant improvement over the low-pass filter for all three degradation metrics and energy consumption. In particular, the low-pass filter provides almost no reduction in peak battery power due to the instances where a high positive power is demanded and the supercapacitor is empty, whereas the optimal controller reduces the peak battery power from 113.3 kW to 36.4 kW. This clearly demonstrates the benefits of a predictive controller that can ensure the supercapacitor has sufficient charge available for high power events. Furthermore, the RMS battery power and total battery throughput are also significantly reduced using the optimal controller, and the supercapacitor state constraints are only active for a small fraction of the journey. This suggests that the filter is inherently sub-optimal, even when detrimental control decisions are not being forced by the state constraints.

\begin{table}
\begin{center}
\caption{Approximate measures of battery degradation for the trajectories shown in Figure \ref{figure::results2}. The percentage values are the improvements relative to the all-battery baseline.}
\label{table::results1}
\begin{tabular}{c | c c c }
 & All Battery & Low-pass Filter & Optimal \\ \hline
RMS($u$) (kW) & 31.8 & 25.2 (-20.8\%) & 21.0 (-33.9\%) \\
 $\max |u|$ (kW) & 133.3 & 132.1 (-0.9\%) & 36.4 (-72.7\%) \\
 $\sum |u|$ (MJ) & 17.9 & 15.9 (-11.3\%) & 13.6 (-24.2\%) \\
 $\sum (u + v)$ (MJ) & 6.8  & 6.6 (-3.7\%) & 6.4 (-5.5\%)
\end{tabular}
\end{center}
\end{table}

Figure \ref{figure::results1} shows the RMS battery power, peak battery power, battery power throughput, and energy consumption for each control method on all 49 journeys, and the averages are summarized in Table \ref{table::Controller_results}. It can be clearly seen that the optimal controller provides a significant and consistent improvement over the low-pass filter across all four metrics and every journey. 

High temperatures and thermal gradients have been identified as factors contributing to battery degradation \cite[\S 3]{TOMASZEWSKA2019100011}, and the optimal controller provides the greatest overall reduction in peak current, which will have an impact on reducing both battery temperature (there will be a lag between heat being generated within individual cells and being sensed by the cooling system) and temperature gradients within each cell (the increased temperature will initially be localised to the core and/or terminals). Conversely, for a significant number of journeys the low-pass filter provides no perceptible reduction in peak battery power, as shown for the journey in Figure~\ref{figure::results2}. The optimal controller also significantly reduces the RMS battery power, which will also reduce thermal degradation, and total power throughput, which will reduce degradation from battery cycling. Finally, the optimal controller also increases the efficiency of the powertrain over the low-pass filter, from a reduction of 3.8\% to 5.7\% relative to the all battery mode, providing an increase in the range available to the vehicle from full charge.

\begin{figure}
\begin{center}
\includegraphics[scale=1]{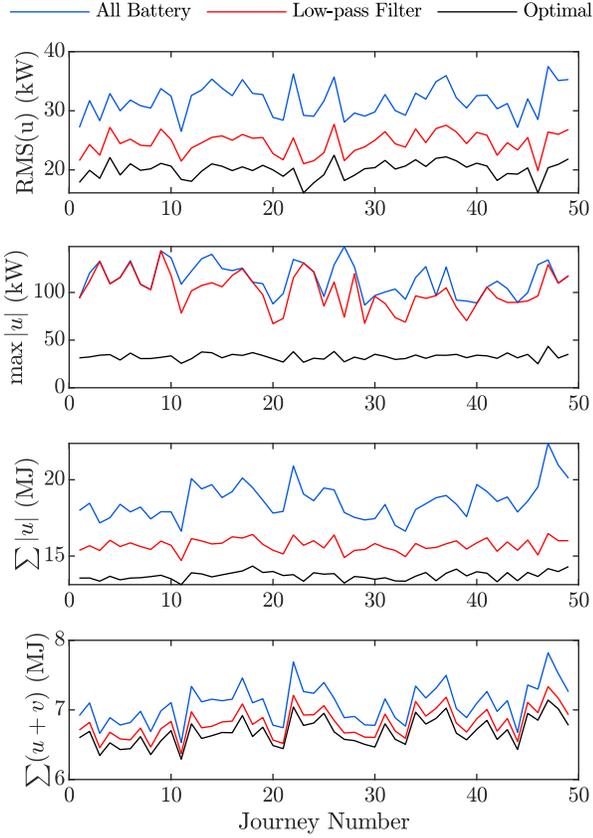}
\caption{Approximate measures of battery degradation for all 49 journeys, using all-battery control, the low-pass filter, and optimal controls.  \label{figure::results1}}
\end{center}
\end{figure} 

\begin{table}
\begin{center}
\caption{Averages of the approximate measures of battery degradation for all 49 journeys and all control methods. The stated percentages are the average of the percentage improvements relative to the all-battery baseline.}
\label{table::Controller_results}
\begin{tabular}{c | c c c }
 & All Battery & Low-Pass Filter & Optimal \\ \hline
$\overline{\text{RMS}(u)}$ (kW) & 31.7 & 24.6 (-22.5\%) & 20.0 (-36.8\%) \\
 $\overline{\max | u |}$ (kW) & 114.1 & 101.5 (-11.0\%) & 32.7 (-71.4\%) \\
 $\overline{\sum | u |}$ (MJ) & 18.6 & 15.7 (-15.6\%) & 13.7 (-26.4\%) \\
 $\overline{\sum (u + v)}$ (MJ) & 7.1 & 6.8 (-3.8\%) & 6.7 (-5.7\%)
\end{tabular}
\end{center}
\end{table}

Figure \ref{figure::results4} shows histograms of the horizon lengths of the journeys and the solution times using ADMM and CVX. The average horizon length of $T$ was 815, with a maximum of 1003, for which the average and maximum solution times using CVX were 63$\ $s and 105$\ $s, and the average and maximum solution times using ADMM were 0.38$\ $s and 0.53$\ $s. The ADMM algorithm was implemented sequentially using Matlab code in these experiments, and a compiled implementation where the combined $(u,v)$ updates are performed in parallel will improve the absolute performance further. 

\begin{figure}
\begin{center}
\includegraphics[scale=1]{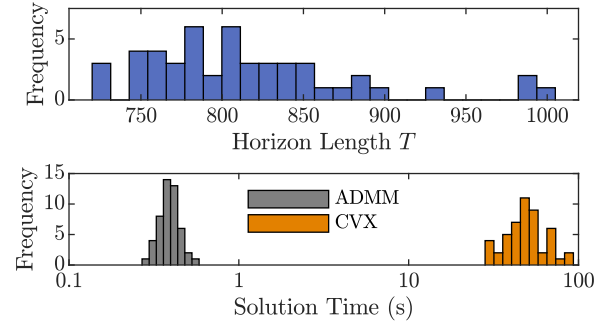}
\caption{Histograms of horizon length, $T$, and solution time using ADMM and CVX for all 49 journeys. \label{figure::results4}}
\end{center}
\end{figure} 

These results have important implications for electric vehicle powertrain control and design. The speed of computation suggests that the ADMM algorithm is a promising candidate for a real-time, online, receding-horizon MPC implementation, which could significantly reduce the battery degradation and energy consumption characteristics of a given electric powertrain design. Furthermore, the algorithm could also be used to determine the optimal size of the powertrain components (a problem considered in \cite{Song2015} and \cite{Masih-Tehrani2013}). In this case, the speed of computation could allow a brute-force approach, where every possible combination of a discrete set of powertrain parameters is evaluated against a set of candidate drive-cycles.

One of the limitations of the proposed approach is that it does not address uncertainty in the predictions of driver behaviour, and it is possible that sufficiently inaccurate predictions could be generated so that the performance of optimization-based controllers becomes worse than a low-pass filter. It is, however, worth highlighting that the convex formulation permits the use of scenario MPC (e.g. \cite{SCHILDBACH20143009}) to explicitly consider the uncertainty in future driver behaviour by optimizing over multiple predictions, possibly taken from samples of previous velocity trajectories. A systematic investigation of the robustness of the proposed method to prediction errors is left for future work. 

\section{Conclusions}
This paper proposes a convex optimization framework for optimal power allocation in electric vehicles with a battery/supercapacitor hybrid storage system, considering hard power limits on all storage and powertrain components and hard limits on battery and supercapacitor storage. An ADMM algorithm is proposed for the solution of the resulting convex optimization problem that has $\mathcal{O} (T)$ computation requirement for each variable update (and which is suitable for parallelization), and $\mathcal{O}(T)$ memory requirement. The convex formulation is compared in simulations with a low-pass filtering heuristic against an all-battery baseline, and it was shown that the optimal controller significantly reduced several measures of battery degradation. Finally, it was demonstrated that the ADMM algorithm solved the optimal control problem in less than a second, even with a horizon of 1000 samples.

\section*{Acknowledgment}

We would like to thank Professor Paul Goulart and Dr Ross Drummond from the University of Oxford, UK, for helpful contributions during the planning and development of this publication. 

We would also like to thank Professor Luigi del Re and Philipp
Polterauer at Johannes Kepler University Linz, Austria, for
generously sharing the driver data used in this publication and assisting with its processing.


%

\appendices

\section{Analysis of Problem (\ref{equation::optimal_control_problem})} \label{appendix::analysis_of_optimization_problem}
The Lagrangian equation for problem (\ref{equation::optimal_control_problem}) is
\begin{multline*}
\mathcal{L}(u,v) := u + v - \lambda_1^\top ( u - \underline{u} ) - \lambda_2^\top (\overline{u} - u ) - \lambda_3^\top ( v - \overline{v} ) \\
- \lambda_4^\top ( \overline{v} - v ) - \lambda_5^\top ( \mathbf{1} x_0 - \Psi u - \underline{x} ) - \lambda_6^\top ( \overline{x} - \mathbf{1} x_0 + \Psi u ) \\
- \lambda_7^\top ( \mathbf{1} y_0 - \Psi v - \underline{y} ) - \lambda_8^\top ( \overline{y} - \mathbf{1} y_0 + \Psi v )  \\
- \lambda_9^\top ( g(u) + f(v) - \overline{e} ) - \lambda_{10}^\top ( \overline{e} - g(u) - f(v) ),
\end{multline*}
where 
\begin{align*}
g(u) &: = (g_0 ( u_0 ), \dots , g_{T-1} ( u_{T-1} )) \\  f(v) &: = (f_0 ( v_0 ), \dots , f_{T-1} ( v_{T-1} ) ).
\end{align*}
The first-order necessary conditions for optimality imply ${\lambda_j^\star \geq 0}$ $\forall j \in \{0, \dots, 10 \}$ and $\nabla_u \mathcal{L} = 0$, $\nabla_v \mathcal{L} = 0$, so that
\begin{align*}
\lambda_{1,t}^\star - \lambda_{2,t}^\star + \sum_{i=t}^{T-1} [ \lambda_{5,i}^\star - \lambda_{6,i}^\star] + g'_t(u_t^\star ) ( \lambda_{10,t}^\star - \lambda_{9,t}^\star ) & = -1 \\ 
\lambda_{3,t}^\star - \lambda_{4,t}^\star + \sum_{i=t}^{T-1} [ \lambda_{7,i}^\star - \lambda_{8,i}^\star ] + f'_t(v_t^\star ) ( \lambda_{10,t}^\star - \lambda_{9,t}^\star ) & = -1
\end{align*}
$\forall t \in \{0,\dots,T-1 \}$. Suppose that the constraint $\overline{e}_t - g_t (u_t) - f_t ( v_t ) \geq 0$ is active for some $t$, then $\lambda_{10,t} > 0$, $\lambda_{9,t} = 0$, and
\begin{subequations}
\begin{align}
g'_t (u^\star_t ) \lambda_{10,t}^\star = -1 + \lambda_{1,t}^\star - \lambda_{2,t}^\star - \sum_{i=t}^{T-1} [ \lambda_{5,i}^\star - \lambda_{6,i}^\star ],\label{equation::KKTu} \\
f'_t (v^\star_t ) \lambda_{10,t}^\star = -1 + \lambda_{3,t}^\star - \lambda_{4,t}^\star - \sum_{i=t}^{T-1} [ \lambda_{7,i}^\star - \lambda_{8,i}^\star ]. \label{equation::KKTv}
\end{align}
\end{subequations}
Assumption \ref{assumption::increasing} implies that $g'_t (u^\star ) > 0$ and $f'_t (v^\star ) > 0$, so $g'_t ( u^\star_t ) \lambda_{10,t}^\star > 0$ and $f'_t ( v^\star_t ) \lambda_{10,t}^\star > 0$. In the case where the bounds on battery and supercapacitor energy and power are not considered (i.e. $\lambda_j^\star=0 \ \forall j \in \{1,\dots,6\}$), this contradicts (\ref{equation::KKTu}) and (\ref{equation::KKTv}) so the constraint $\overline{e}_t - g_t (u_t) - f_t ( v_t ) \geq 0$ cannot be active at the solution. In the case where these bounds are included in the problem formulation, the same conclusion cannot be reached as $\lambda_{1,t}^\star$ and/or $\lambda_{6,i}^\star$ for some $i \in \{t, \dots, {T-1} \}$ may be nonzero.

\section{Piecewise Constant Solution of  (\ref{equation::optimal_control_problem_convex_2})} \label{appendix::analysis_of_optimization_problem2}

A similar approach to that used in Appendix \ref{appendix::analysis_of_optimization_problem} can be used to show that, under the assumption that the constraints on battery power are inactive (i.e. $\lambda_{1,t}=0$ and $\lambda_{2,t}=0$ $\forall t$) and that the upper constraint on powertrain power is inactive (i.e. $\lambda_{10,t}=0 \ \forall t$), then the solution to (\ref{equation::optimal_control_problem_convex}) satisfies
\begin{subequations}
\begin{align}
- g'_t (u^\star_t ) \lambda_{9,t}^\star = -1 - \sum_{i=t}^{T-1} [ \lambda_{5,i}^\star - \lambda_{6,i}^\star ],\label{equation::KKTu2} \\
- f'_t (v^\star_t ) \lambda_{9,t}^\star  = -1 - \sum_{i=t}^{T-1} [ \lambda_{7,i}^\star - \lambda_{8,i}^\star ]. \label{equation::KKTv2}
\end{align}
\end{subequations}
Consider a set $\{\underline{t}, \dots, \overline{t} \}$ where the upper and lower bounds on $x$ and $y$ are inactive, i.e. $\lambda_{5,t}^\star = 0$, $\lambda_{6,t}^\star = 0$, $\lambda_{7,t}^\star = 0$, and $\lambda_{8,t}^\star = 0$ $\forall t \in \{\underline{t}, \dots, \overline{t} \}$, then $-1 - \sum_{i=t}^{T-1} [ \lambda_{5,i}^\star - \lambda_{6,i}^\star ] = c_1$ and $-1 - \sum_{i=t}^{T-1} [ \lambda_{7,i}^\star - \lambda_{8,i}^\star ] = c_2$ $\forall t \in \{\underline{t}, \dots, \overline{t} \}$. Therefore, $u_t^\star$, $v_t^\star$, and $\lambda_{9,t}^\star$ are given by
\begin{align*}
- g'_t (u^\star_t ) \lambda_{9,t}^\star = c_1 \quad \textrm{and} \quad -f'_t (v^\star_t ) \lambda_{9,t}^\star = c_2.
\end{align*}
For the system models presented in Section \ref{section::numerical_experiments}, $f'_t(v_t) =1 $ and $g_t'(u_t) = 1 - \frac{2R u _t}{V^2} \ \forall t \in \{0,\dots, T-1 \} $, so $\lambda_{9,t}^\star = -c_2 \ \forall t \in \{ \underline{t}, \dots , \overline{t} \}$, and
$$
u_t = \frac{V^2}{2 R} \left(1 - \frac{c_1}{c_2} \right) \ \forall t \in \{ \underline{t}, \dots , \overline{t} \},
$$
which implies that the optimal control input is constant on the interval $\{ \underline{t},\dots,\bar{t} \}$.

\section{Combined $u$ and $v$ update} \label{appendix::combined_u_v_algorithm}
The optimization problem in (\ref{subequation::u_v_update}) is equivalent to
\begin{equation}\label{equation::u_v_update2}
\begin{aligned}
\underset{(u_t , v_t )}{\textrm{argmin}} \ & \frac{\rho_1}{2} (u_t - \zeta_t^{(j)} + \lambda_{1,t}^{(j)})^2  + \frac{\rho_2}{2} ( v_t - \eta_t^{(j)} + \lambda_{2,t}^{(j)})^2 \\
\text{s.t.} \ &  \underline{u}_t \leq u_t \leq \overline{u}_t, \quad \underline{v}_t \leq v_t \leq \overline{v}_t, \\
& \hat{e}_t \leq g_t(u_t) + f_t (v_t), \quad \hat{g}_t (u_t) + \hat{f}_t (v_t ) \leq \overline{e}_t,
\end{aligned}
\end{equation}
which is a convex (quadratic) inequality constrained optimization problem, and Figure \ref{figure::set_C_t} shows the constraint set for an illustrative example. A rigorous treatment of the proposed approach is provided below, but the principle is that there are three candidate solutions that can be obtained from simpler optimization problems: (a) problem (\ref{equation::u_v_update2}) with the constraints $\hat{e}_t \leq g_t(u_t) + f_t (v_t)$ and $\hat{g}_t (u_t) + \hat{f}_t (v_t ) \leq \overline{e}_t$ discarded, (b) problem (\ref{equation::u_v_update2}) with equality constraint $\hat{g}_t (u_t) + \hat{f}_t (v_t ) = \overline{e}_t$, (c) problem (\ref{equation::u_v_update2}) with equality constraint $\hat{e}_t = g_t(u_t) + f_t (v_t)$. Problems (a) and (b) have analytical solutions, and (c) reduces to a one-dimensional problem that can be solved numerically. Each of these cases will now be considered in detail; it is assumed throughout that (\ref{equation::u_v_update2}) is feasible.
\begin{figure}
\begin{center}
\begin{tikzpicture}
\def\uhat{-1.5E5}
\def\ghat{\uhat - \R*\uhat*\uhat/\V^2}
\def\gdashhat{1 - 2*\R*\uhat/\V^2}
\def\a{\R/\V/\V}
\def\b{\gdashhat - 1}
\def\c{-1.5E5 - (1.0E5) + \ghat - (\gdashhat)*(\uhat)}
\def\xright{(-(\b) + sqrt((\b)^2 - 4*(\a)*(\c)))/(2*(\a))}
\def\yright{-(\xright - \R*(\xright)^2/\V^2) - 1.5E5}
\def\xleft{(-(\b) - sqrt((\b)^2 - 4*(\a)*(\c)))/(2*(\a))}
\def\yleft{-(\xleft - \R*(\xleft)^2/\V^2) - 1.5E5}
\def\R{0.1}
\def\V{300}
\begin{axis}[width = 220pt, height = 185pt, axis lines = center, enlargelimits = true, clip = false, xlabel = {$u_t$}, ylabel = {$v_t$}, xmajorticks = false, xticklabels = {,,}, ytick = {0.5}, yticklabels = {$g_k(P_{drv,k})$},>=latex]
\addplot[domain=-8E5:4.5E5,<-] {-(300^2 - (300 - (2*0.1*x/300))^2)/(4*0.1) - 1.5E5} node [pos=0, above] {$\hat{e}_t = g_t(u_t) + f_t (v_t)$};
\addplot[domain=-6.5E5:3.5E5, dashed] {-(x - \R*x^2/(\V^2)) + 1.0E5};
\addplot[domain=-6E5:1E5] {6E5} node [pos=1, right] {$v_t = \overline{v}_t$};
\addplot[domain=-5.0E5:3.5E5] {-2.5E5}  node [pos=0, left] {$v_t = \underline{v}_t$};
\addplot[domain=-4.5E5:0.5E5, name path=C, opacity=0] {6E5};
\addplot[domain=-4.5E5:0.5E5, name path=D, opacity=0] {-2.5E5};
\addplot [mark=none, name path=E] coordinates {(-4.5E5, 8.5E5) (-4.5E5, -3.5E5)} node [pos=1, below] {$u_t = \underline{u}_t$};
\addplot [mark=none, name path=F] coordinates {(0.5E5, 7E5) (0.5E5, -3.5E5)};
\addplot [color=black,mark=none] coordinates {(1.5E5, 7E5)} node [above] {$u_t = \overline{u}_t$};
\addplot[domain=-7.4E5:5E5,->] {-(\ghat) - (\gdashhat)*(x - \uhat) + 1E5} node [pos=1, below] {$\hat{g}_t (u_t) + \hat{f}_t ( v_t ) = \overline{e}_t$};
\addplot[domain={\xleft}:{\xright}, name path =B, opacity=0] {-(\ghat) - (\gdashhat)*(x - \uhat) + 1E5} ;
\addplot[domain={\xleft}:{\xright}, name path=A, opacity=0] {-(x - \R*x^2/(\V^2)) - 1.5E5};
\addplot [red, opacity=0.15] fill between[of=A and B];
\addplot [green, opacity=0.1] fill between[of=C and D];
\addplot [green, opacity=0] fill between[of=E and F];
\end{axis}
\end{tikzpicture}
\caption{Illustration of constraint sets defined by $\mathcal{C}_t$ (in red) and $\mathcal{U}_t \cap \mathcal{V}_t$ (in green).  \label{figure::set_C_t}}
\bigskip
\begin{tikzpicture}
\def\opac{0.5}
\def\opaccol{1}
\def\uhat{-1.5E5}
\def\ghat{\uhat - \R*\uhat*\uhat/\V^2}
\def\gdashhat{1 - 2*\R*\uhat/\V^2}
\def\a{\R/\V/\V}
\def\b{\gdashhat - 1}
\def\c{-1.5E5 - (1.0E5) + \ghat - (\gdashhat)*(\uhat)}
\def\xright{(-(\b) + sqrt((\b)^2 - 4*(\a)*(\c)))/(2*(\a))}
\def\yright{-(\xright - \R*(\xright)^2/\V^2) - 1.5E5}
\def\xleft{(-(\b) - sqrt((\b)^2 - 4*(\a)*(\c)))/(2*(\a))}
\def\yleft{-(\xleft - \R*(\xleft)^2/\V^2) - 1.5E5}
\def\R{0.1}
\def\V{300}
\definecolor{blu}{RGB}{0,127,255}
\definecolor{ora}{RGB}{255,90,0}
\begin{axis}[width = 220pt, height = 185pt, axis lines = center, enlargelimits = true, clip = false, xlabel = {$u_t$}, ylabel = {$v_t$}, xmajorticks = false, xticklabels = {,,}, ytick = {0.5}, yticklabels = {$g_k(P_{drv,k})$},>=latex]
\addplot[domain=-8E5:-4.5E5,<-, opacity=\opac] {-(300^2 - (300 - (2*0.1*x/300))^2)/(4*0.1) - 1.5E5} node [pos=0, above] {$\hat{e}_t = g_t(u_t) + f_t (v_t)$};
\addplot[domain=0.5E5:4.5E5, opacity=\opac] {-(300^2 - (300 - (2*0.1*x/300))^2)/(4*0.1) - 1.5E5};
\addplot[domain=-6.5E5:3.5E5, dashed, opacity=\opac] {-(x - \R*x^2/(\V^2)) + 1.0E5};
\addplot[domain=-6E5:1E5, opacity=\opac] {6E5} node [pos=1, right] {$v_t = \overline{v}_t$};
\addplot[domain=-5.0E5:3.5E5, opacity=\opac] {-2.5E5}  node [pos=0, left] {$v_t = \underline{v}_t$};
\addplot[domain=-4.5E5:0.5E5, name path=C, opacity=0] {6E5};
\addplot[domain=-4.5E5:0.5E5, name path=D, opacity=0] {-2.5E5};
\addplot [mark=none, name path=E, opacity=\opac] coordinates {(-4.5E5, 8.5E5) (-4.5E5, -3.5E5)} node [pos=1, below] {$u_t = \underline{u}_t$};
\addplot [mark=none, name path=F, opacity=\opac] coordinates {(0.5E5, 7E5) (0.5E5, -3.5E5)};
\addplot [color=black,mark=none, opacity=\opac] coordinates {(1.5E5, 7E5)} node [above] {$u_t = \overline{u}_t$};
\addplot[domain=-7.4E5:5E5,->, opacity=\opac] {-(\ghat) - (\gdashhat)*(x - \uhat) + 1E5} node [pos=1, below] {$\hat{g}_t (u_t) + \hat{f}_t ( v_t ) = \overline{e}_t$};
\addplot[ora, mark =*, mark options={fill=blu, line width=1pt}, opacity=\opaccol] coordinates {(\xleft, \yleft )};
\addplot[ora, mark =*, mark options={fill=blu, line width=1pt}, opacity=\opaccol] coordinates {(\xright, \yright )};
\addplot[ora, mark =*] coordinates {(-4.5E5, {-(-4.5E5 - \R*(4.5E5)^2/(\V^2)) - 1.5E5} )};
\addplot[ora, mark =*] coordinates {(0.5E5, {-(0.5E5 - \R*(0.5E5)^2/(\V^2)) - 1.5E5} )};
\addplot[ora, mark =*, opacity=\opaccol] coordinates {({-4.8675E5}, {6E5} )};
\addplot[ora, mark =*, opacity=\opaccol] coordinates {({1.1459E5}, {-2.5E5} )};
\addplot[ora, opacity=\opaccol, line width=1pt] coordinates {(0.5E5, {-(0.5E5 - \R*(0.5E5)^2/(\V^2)) - 1.5E5 - 2E5} ) (0.5E5, {-(0.5E5 - \R*(0.5E5)^2/(\V^2)) - 1.5E5 + 1E5} )};
\addplot[ora, opacity=\opaccol, line width=1pt] coordinates {(-4.5E5, {-(-4.5E5 - \R*(4.5E5)^2/(\V^2)) - 1.5E5-2E5} ) (-4.5E5, {-(-4.5E5 - \R*(4.5E5)^2/(\V^2)) - 1.5E5 + 2E5} )};
\addplot[blu, mark =*, opacity=\opaccol] coordinates {(-4.5E5,{-(\ghat) - (\gdashhat)*(-4.5E5 - \uhat) + 1E5})};
\addplot[blu, mark =*] coordinates {(0.5E5,{-(\ghat) - (\gdashhat)*(0.5E5 - \uhat) + 1E5})};
\addplot[blu, mark =*, opacity=\opaccol] coordinates {(2.4375E5,-2.5E5)};
\addplot[blu, mark =*] coordinates {(-3.9375E5,6E5)};
\addplot[blu, opacity=\opaccol, line width=1pt] coordinates {(-3.9375E5,4E5) (-3.9375E5,8E5)};
\addplot[blu, , opacity=\opaccol, line width=1pt] coordinates {(0.5E5,{-(0.5E5 - \R*(0.5E5)^2/(\V^2)) - 1.5E5 + 1E5}) (0.5E5,{-(\ghat) - (\gdashhat)*(0.5E5 - \uhat) + 1E5 + 2E5})};
\addplot[blu, domain={\xleft}:{\xright}, name path =B, opacity=0] {-(\ghat) - (\gdashhat)*(x - \uhat) + 1E5} ;
\addplot[domain={\xleft}:{\xright}, name path =B, opacity=0] {-(\ghat) - (\gdashhat)*(x - \uhat) + 1E5} ;
\addplot[domain={\xleft}:{\xright}, name path=A, opacity=\opac] {-(x - \R*x^2/(\V^2)) - 1.5E5};
\addplot[black, domain={-4.5E5}:{0.5E5}, line width=1pt] {-(x - \R*x^2/(\V^2)) - 1.5E5};
\addplot[black, domain={-3.9375E5}:{0.5E5}, opacity=1, line width=1pt] {-(\ghat) - (\gdashhat)*(x - \uhat) + 1E5} ;
\addplot [black, opacity=0.05] fill between[of=A and B];
\addplot [black, opacity=0.05] fill between[of=C and D];
\addplot [green, opacity=0] fill between[of=E and F];
\end{axis}
\end{tikzpicture}
\caption{Illustration of constraint sets for reduced problems (\ref{equation::reduced_problem_1}) and (\ref{equation::reduced_problem_2}).  \label{figure::set_C_t_new}}
\end{center}
\end{figure}

(a) The constraints $\hat{e}_t \leq g_t(u_t) + f_t (v_t)$ and $\hat{g}_t (u_t) + \hat{f}_t (v_t ) \leq \overline{e}_t$ (i.e. $(u_t, v_t) \in \mathcal{C}_t$) are discarded, and a candidate solution $(u^{\dagger 1}_t,v^{\dagger 1}_t)$ to (\ref{equation::u_v_update2}) is obtained from
\begin{equation*}
\begin{aligned}
\underset{(u_t , v_t )}{\textrm{argmin}} \ & \frac{\rho_1}{2} (u_t - \zeta_t^{(j)} + \lambda_{1,t}^{(j)})^2  + \frac{\rho_2}{2} ( v_t - \eta_t^{(j)} + \lambda_{2,t}^{(j)})^2 \\
\text{s.t.} \ &  \underline{u}_t \leq u_t \leq \overline{u}_t, \quad \underline{v}_t \leq v_t \leq \overline{v}_t, \\
\end{aligned}
\end{equation*} 
as $u_t^{\dagger 1} = \min \{ \overline{u}_t , \max \{ \underline{u}_t, \zeta_t^{(j)} - \lambda_{1,t}^{(j)} \} \} $ and $v_t^{\dagger 1} = \min \{ \overline{v}_t ,  \max \{ \underline{u}_t , \eta_t^{(j)} - \lambda_{2,t}^{(j)} \} \}$. If the discarded constraints are satisfied by the candidate solution, then this is the actual solution to $(\ref{equation::u_v_update2})$. If not, then the solution must be further constrained by $\hat{e}_t = g_t(u_t) + f_t (v_t)$ and/or $\hat{g}_t (u_t) + \hat{f}_t (v_t ) = \overline{e}_t$. Therefore, two further candidate solutions are obtained by tightening each of $\hat{e}_t \leq g_t(u_t) + f_t (v_t)$ and $\hat{g}_t (u_t) + \hat{f}_t (v_t ) \geq \overline{e}_t$ in (\ref{equation::u_v_update2}) to equality constraints. 

(b) A candidate solution $(u^{\dagger 2}_t,v^{\dagger 2}_t)$ is obtained from
\begin{equation}\label{equation::opt_b}
\begin{aligned}
\underset{(u_t , v_t )}{\textrm{argmin}} \ & \frac{\rho_1}{2} (u_t - \zeta_t^{(j)} + \lambda_{1,t}^{(j)})^2  + \frac{\rho_2}{2} ( v_t - \eta_t^{(j)} + \lambda_{2,t}^{(j)})^2 \\
\text{s.t.} \ &  \underline{u}_t \leq u_t \leq \overline{u}_t, \quad \underline{v}_t \leq v_t \leq \overline{v}_t, \\
& \hat{e}_t \leq g_t(u_t) + f_t (v_t), \quad \hat{g}_t (u_t) + \hat{f}_t (v_t ) = \overline{e}_t.
\end{aligned}
\end{equation}

In section \ref{section::ADMM} it was specified that $\hat{u}$ and $\hat{v}$ were chosen so that $g'(\hat{u} )$ and $f'(\hat{v})$ are nonzero (and under Assumption \ref{assumption::increasing} must be positive), so $\hat{g}^{-1}$ and $\hat{f}^{-1}$ exist and are both affine and increasing.

\begin{proposition}\label{claim::claim1}
Define the set 
$$
\hat{\mathcal{U}}_t := \{u_t \in \mathbb{R} : \hat{f}_t^{-1} (\overline{e}_t - \hat{g}_t ( u_t)) \geq f_t^{-1} (\hat{e}_t - g_t ( u_t)) \},
$$ 
and define $u_t^{\cap 1} : = \inf \hat{\mathcal{U}}_t$ and $u_t^{\cap 2} : = \sup \hat{\mathcal{U}}_t$. Then 
\begin{align*}
& \{ (u_t, v_t) \in \mathbb{R}^2 : \hat{g}_t (u_t) + \hat{f}_t (v_t ) = \overline{e}_t, \hat{e}_t \leq g_t(u_t) + f_t (v_t)  \} \\
= & \{ (u_t, v_t) \in \mathbb{R}^2 : \hat{g}_t (u_t) + \hat{f}_t (v_t ) = \overline{e}_t, u_t \in [u_t^{\cap 1}, u_t^{\cap 2} ] \cap \mathbb{R} \}.
\end{align*}
\end{proposition}
\begin{proof}
The set $\{(u_t,v_t) \in \mathbb{R}^2 : \hat{e}_t \leq g_t(u_t) + f_t (v_t) \}$ is convex under Assumption \ref{assumption::concave} and is the epigraph of the function $v_t =  f_t^{-1} (\hat{e}_t - g_t ( u_t))$, which therefore must also be a convex function. Additionally, $\hat{g}_t (u_t) + \hat{f}_t (v_t ) = \overline{e}_t \Leftrightarrow v_t = \hat{f}_t^{-1} (\overline{e}_t - \hat{g}_t ( u_t))$, so
\begin{align*}
& \{ (u_t, v_t) \in \mathbb{R}^2 : \hat{g}_t (u_t) + \hat{f}_t (v_t ) = \overline{e}_t, \hat{e}_t \leq g_t(u_t) + f_t (v_t)  \} \\
= & \{ (u_t, v_t) \in \mathbb{R}^2 : \hat{g}_t (u_t) + \hat{f}_t (v_t ) = \overline{e}_t, \\
& \quad \quad \quad \quad \hat{f}_t^{-1} (\overline{e}_t - \hat{g}_t ( u_t)) \geq f_t^{-1} (\hat{e}_t - g_t ( u_t))  \}.
\end{align*}
Therefore, $\hat{\mathcal{U}}_t$ defines the values of $u_t$ where an affine function is greater than or equal to a convex function, so is convex. 

Consider the illustration of $\hat{g}_t (u_t) + \hat{f}_t (v_t ) = \overline{e}_t$ and $\hat{e}_t = g_t(u_t) + f_t (v_t)$ in Figure \ref{figure::set_C_t}, and assume that the function $f_t^{-1} (\hat{e}_t - g_t ( u_t))$ is strongly convex (which is the case for the models specified in Sections \ref{subsection::battery_model} and \ref{subsection::supercap_model}). This implies that $\mathcal{U}_t$ is closed and $\hat{\mathcal{U}}_t = [u_t^{\cap 1}, u_t^{\cap 2} ]$ (i.e. $u_t^{\cap 1} \in \mathbb{R}$ and $u_t^{\cap 2} \in \mathbb{R}$ are the `lower' and `upper' intersection points of the functions $\hat{g}_t (u_t) + \hat{f}_t (v_t ) = \overline{e}_t$ and $\hat{e}_t = g_t(u_t) + f_t (v_t)$). 

Now assume that $f_t^{-1} (\hat{e}_t - g_t ( u_t))$ is not strongly convex. Under the assumptions on $f_t(\cdot)$ and $g_t(\cdot)$ it is possible to construct cases in which there are no intersection points, or only an `upper' or `lower' intersection point (e.g. if $f_t(\cdot)$ and $g_t(\cdot)$ are piecewise affine). In these cases $\hat{\mathcal{U}}_t = [u_t^{\cap 1}, u_t^{\cap 2} ]$ where $u^{\cap 1} \in \{ -\infty, \mathbb{R} \} $ and $u^{\cap 2} \in \{ \mathbb{R},  \infty \}$. 
\end{proof}

\begin{remark}\label{remark::remark1}
If $\hat{e}_t = g_t(u_t) + f_t (v_t)$ is strongly convex and twice continuously differentiable (which is the case for the models specified in Sections \ref{subsection::battery_model} and \ref{subsection::supercap_model}), then $\overline{e}_t = \hat{e}_t$ implies $u_t^{\cap 1} = u_t^{\cap 2}$, and $\mathcal{C}_t$ has a single element that is trivially the solution to (\ref{equation::u_v_update2}). This also implies that $u^{\cap 1}_t = u^\star_t$, so the linearization points $\hat{u}_t$ and $\hat{v}_t$ should 
be chosen so that the power consumed by the powertrain is zero when $\overline{e}_t = \hat{e}_t$, as this occurs when the vehicle is stationary.
\end{remark}

\begin{proposition}\label{claim::claim2}
\begin{align*}
& \{ (u_t, v_t ) \in \mathbb{R}^2 : \hat{g}_t (u_t) + \hat{f}_t (v_t ) = \overline{e}_t, a \leq v_t \leq b\} \\
= \  & \{ (u_t, v_t ) \in \mathbb{R}^2 : \hat{g}_t (u_t) + \hat{f}_t (v_t ) = \overline{e}_t, \\
& \quad \quad \quad \hat{g}^{-1}_t ( \overline{e}_t - \hat{f}_t ( b) ) \leq u_t \leq \hat{g}^{-1}_t ( \overline{e}_t - \hat{f}_t ( a) )\}
\end{align*}
\end{proposition}
\begin{proof}
In section \ref{section::ADMM} it was specified that $\hat{u}$ and $\hat{v}$ were both chosen so that $g'(\hat{u} )$ and $f'(\hat{v} )$ were non-zero (which under Assumption \ref{assumption::increasing} implies that they are greater than zero), so $\hat{f}_t(\cdot)$ and $\hat{g}_t(\cdot)$ are both increasing, and it can then be shown that
\begin{align*}
v_t \leq b \Rightarrow & \ \overline{e}_t - \hat{f}_t ( v_t ) \geq \overline{e}_t - \hat{f}_t ( b) \\
\Rightarrow & \ u_t \geq \hat{g}^{-1}_t ( \overline{e}_t - \hat{f}_t ( b) ).
\end{align*}
It can similarly be shown that $u_t \leq \hat{g}_t^{-1} ( \overline{e}_t - \hat{g}_t ( a))$.
\end{proof}
Propositions \ref{claim::claim1} and \ref{claim::claim2} imply that (\ref{equation::opt_b}) is equivalent to
\begin{equation}\label{equation::reduced_problem_1}
\begin{aligned}
\underset{u_t }{\textrm{argmin}} \ & \frac{\rho_1}{2} (u_t - \zeta_t^{(j)} + \lambda_{1,t}^{(j)})^2 \\
& \quad + \frac{\rho_2}{2} ( \hat{f}_t^{-1} ( \overline{e}_t - \hat{g}_t ( u_t )) - \eta_t^{(j)} + \lambda_{2,t}^{(j)})^2 \\
\text{s.t.} \ &  u_t \geq \max \{ \underline{u}_t, u_t^{\cap 1}, \hat{g}^{-1}_t ( \overline{e}_t - \hat{f}_t ( \overline{v}_t ) ) \} \\
& u_t \leq \min \{ \overline{u}_t, u_t^{\cap 2}, \hat{g}^{-1}_t ( \overline{e}_t - \hat{f}_t ( \underline{v}_t ) ) \} 
\end{aligned}
\end{equation}
which is a one-dimensional constrained quadratic minimization problem for $\hat{u}_t$, with the constraint set illustrated in blue in Figure \ref{figure::set_C_t_new}. The solution to the unconstrained problem can be obtained analytically and projected onto the upper and lower bounds on $\hat{u}_t$ to obtain $u^{\dagger 2}_t$, then the corresponding value of $v_t$ can then be returned from $v^{\dagger 2}_t = \hat{f}_t^{-1} ( \overline{e}_t - \hat{g}_t ( u^{\dagger 2}_t ))$. 

(c) 
The final candidate solution $(u^{\dagger 3}_t, v^{\dagger 3}_t )$ is obtained from 
\begin{equation*}
\begin{aligned}
\underset{(u_t , v_t )}{\textrm{argmin}} \ & \frac{\rho_1}{2} (u_t - \zeta_t^{(j)} + \lambda_{1,t}^{(j)})^2  + \frac{\rho_2}{2} ( v_t - \eta_t^{(j)} + \lambda_{2,t}^{(j)})^2 \\
\text{s.t.} \ &  \underline{u}_t \leq u_t \leq \overline{u}_t, \quad \underline{v}_t \leq v_t \leq \overline{v}_t, \\
& \hat{e}_t = g_t(u_t) + f_t (v_t), \quad \hat{g}_t (u_t) + \hat{f}_t (v_t ) \leq \overline{e}_t,
\end{aligned}
\end{equation*}
which can be shown to be equivalent to
\begin{equation}\label{equation::reduced_problem_2}
\begin{aligned}
\underset{u_t}{\textrm{argmin}} \ & \frac{\rho_1}{2} (u_t - \zeta_t^{(j)} + \lambda_{1,t}^{(j)})^2  \\
& \quad + \frac{\rho_2}{2} ( f_t^{-1} ( \overline{e}_t - g_t ( u_t )) - \eta_t^{(j)} + \lambda_{2,t}^{(j)})^2 \\
\text{s.t.} \ &  u_t \geq \max \{ \underline{u}_t, u_t^{\cap 1}, \hat{g}^{-1}_t ( \overline{e}_t - \hat{f}_t ( \overline{v}_t ) ) \} \\
& u_t \leq \min \{ \overline{u}_t, u_t^{\cap 2}, \hat{g}^{-1}_t ( \overline{e}_t - \hat{f}_t ( \underline{v}_t ) ) \}
\end{aligned}
\end{equation}
using the same approach as for (b). The constraint set for (\ref{equation::reduced_problem_2}) is illustrated in red in Figure \ref{figure::set_C_t_new}. The function $f_t^{-1} ( \overline{e}_t - g_t ( u_t ))$ is nonlinear, so the cost function in (\ref{equation::reduced_problem_2}) is nonconvex in general, but a stationary point can be obtained without the inequality constraints using an iterative algorithm (e.g. Newton's method)\footnote{For the experiments detailed in Section \ref{section::numerical_experiments}, the function $f_t^{-1} ( \overline{e}_t - g_t ( u_t ))$ is quadratic, so the cost function in (\ref{equation::reduced_problem_2}) is quartic, and the stationary points can be found from the roots of a cubic equation.} and then projected onto the bounds on $u_t$. 

The stationary point of (\ref{equation::reduced_problem_2}) and $(u^{\dagger 2}_t, v^{\dagger 2}_t )$ are then evaluated against the cost function of (\ref{equation::u_v_update2}) to determine which is the minimizing argument. If the constraint $\hat{e}_t \leq g_t(u_t) + f_t (v_t)$ is strongly active at the solution, then the stationary point of (\ref{equation::reduced_problem_2}) corresponds to the global minimum of (\ref{equation::u_v_update2}). If the constraint $\hat{e}_t \leq g_t(u_t) + f_t (v_t)$ is not strongly active at the solution to (\ref{equation::u_v_update2}), then problem (\ref{equation::reduced_problem_2}) may have multiple stationary points that may not be minimal for (\ref{equation::u_v_update2}). However, in this case the constraint $\hat{g}_t (u_t) + \hat{f}_t (v_t ) \leq \overline{e}_t$ will be strongly active at the solution to (\ref{equation::u_v_update2}), so $(u^{\dagger 2}_t, v^{\dagger 2}_t )$ will be the minimizing argument of (\ref{equation::u_v_update2}). The pseudocode for update (\ref{subequation::u_v_update}) is presented in Algorithm \ref{algorithm::combined_u_v}.

\begin{algorithm}
\caption{Update (\ref{subequation::u_v_update})\label{algorithm::combined_u_v}}
\begin{algorithmic}[1]
\STATE $u_t^{\dagger 1} \gets \min \{ \overline{u}_t , \max \{ \underline{u}_t, \zeta_t^{(j)} - \lambda_{1,t}^{(j)} \} \} $ 
\STATE $v_t^{\dagger 1} \gets \min \{ \overline{v}_t ,  \max \{ \underline{u}_t , \eta_t^{(j)} - \lambda_{2,t}^{(j)} \} \}$
\IF {$(u^{\dagger 1}_t, v^{\dagger 1}_t ) \in \mathcal{C}_t$}
\STATE $(u_t,v_t)^{j+1} \gets (u^{\dagger 1}_t, v^{\dagger 1}_t )$
\ELSE
\STATE $u^{\dagger 2}_t \gets $ Solution to (\ref{equation::reduced_problem_1})
\STATE $v^{\dagger 2}_t \gets \hat{f}_t^{-1} ( \overline{e}_t - \hat{g}_t ( u^{\dagger 2}_t ))$
\STATE $u^{\dagger 3}_t \gets $ Stationary point of (\ref{equation::reduced_problem_2})
\STATE $v^{\dagger 3}_t \gets \hat{f}_t^{-1} ( \overline{e}_t - \hat{g}_t ( u^{\dagger 3}_t ))$
\STATE Evaluate $(u^{\dagger 2}_t, v^{\dagger 2}_t )$ and $(u^{\dagger 3}_t, v^{\dagger 3}_t ) $ against objective of (\ref{subequation::u_v_update})
\STATE $(u_t,v_t)^{j+1} \gets$ minimizing argument
\ENDIF 
\end{algorithmic}
\end{algorithm}

\section{Matrix Inversion in $\zeta$ and $\eta$ update.} \label{appendix::ADMM_linear_system}

\begin{proposition}
The system of linear equations 
\begin{equation}\label{equation::linear_system}
(kI + \Psi^\top \Psi ) \mathbf{x} = \mathbf{b},
\end{equation}
where $k \in \mathbb{R}_{++}$, $\mathbf{x} \in \mathbb{R}^T$, $\mathbf{b} \in \mathbb{R}^T$, $\Psi$ is a $T \times T$ lower triangular matrix of ones, and $I$ is the identity matrix, can be solved with $\mathcal{O}(T)$ computation and $\mathcal{O}(T)$ memory storage.
\end{proposition}
\begin{proof}
Firstly, note that
$$
\Psi^{-1} = \left[ \begin{smallmatrix}
1 \\ -1 & 1 \\ & \ddots & \ddots \\ & & -1 & 1
\end{smallmatrix} \right], \quad \Psi^{-1} ( \Psi^{-1} )^\top = \left[ \begin{smallmatrix}
1 & -1 \\ -1 & 2 & -1 \\ & \ddots & \ddots & \ddots \\ & & -1 & 2 & -1 \\ & & & -1 & 2
\end{smallmatrix} \right],
$$
and
\begin{align*}
& (kI + \Psi^\top \Psi ) \mathbf{x} = \mathbf{b} \\
\Leftrightarrow \ &  (k \Psi^{-1} (\Psi^{-1})^\top + I ) \mathbf{x} = \Psi^{-1} (\Psi^{-1})^\top \mathbf{b}.
\end{align*}
The matrix $(k \Psi^{-1} (\Psi^{-1})^\top + I )$ is diagonally dominant, so is positive definite from Gershgorin circle theorem. This implies that the Cholesky factorization
$L L^\top = (k \Psi^{-1} (\Psi^{-1})^\top + I )$ exists, where $L$ is a lower diagonal matrix with nonzero entries on the diagonal and first subdiagonal only. The solution to (\ref{equation::linear_system}) can therefore be obtained from
$$
\mathbf{x} = L^\top \text{\textbackslash} L \text{\textbackslash} \Psi^{-1} (\Psi^{-1})^\top \mathbf{b},
$$
where the backslash operator $(\cdot)\text{\textbackslash}(\cdot)$ indicates a backwards/forwards substitution. Multiplications by $\Psi^{-1}$ and $(\Psi^{-1})^\top$ are equivalent to difference operations, and the backwards/forwards substitutions involve bandwidth-2 Cholesky factors, so the total complexity of the operations is $\mathcal{O}(T)$. The only values that need storing in memory are the diagonal and first subdiagonal entries of $L$, which require $\mathcal{O}(T)$ storage.
\end{proof}

\ifCLASSOPTIONcaptionsoff
  \newpage
\fi



\bibliographystyle{IEEEtran}
%
\bibliography{bibl} 

\end{document}